\documentclass[12pt]{amsart}
\usepackage{amscd,amssymb,amsthm,verbatim}
\newcommand{\const}{\operatorname{const.}}

\newcommand{\diam}{\operatorname{diam}}

\newcommand{\dvol}{\operatorname{dvol}}

\newcommand{\Hess}{\operatorname{Hess}}

\newcommand{\Ker}{\operatorname{Ker}}

\newcommand{\R}{{\mathbb R}}

\newcommand{\Ric}{\operatorname{Ric}}

\newcommand{\Rm}{\operatorname{Rm}}

\newcommand{\spann}{\operatorname{span}}

\newcommand{\Tr}{\operatorname{Tr}}

\newcommand{\vol}{\operatorname{vol}}
\newcommand{\Vol}{\operatorname{Vol}}
\newcommand{\Z}{{\mathbb Z}}

\numberwithin{equation}{section}
\theoremstyle{plain}

\newtheorem{lemma}[equation]{Lemma}
\newtheorem{theorem}[equation]{Theorem}

\newtheorem{corollary}[equation]{Corollary}

\theoremstyle{remark}

\newtheorem{remark}[equation]{Remark}
\newtheorem{example}[equation]{Example}
\usepackage{graphicx}
\input{amssym.def}
\input{amssym}
\input{epsf}
\usepackage{a4wide}

\begin{document}

\title[On scalar curvature lower bounds and scalar curvature measure]{On scalar curvature lower bounds and scalar curvature measure}

\author{John Lott}
\address{Department of Mathematics\\
University of California, Berkeley\\
Berkeley, CA  94720-3840\\
USA} \email{lott@berkeley.edu}

\thanks{Research partially supported by NSF grant
DMS-1810700}
\date{July 14, 2022}
\begin{abstract}
  We relate the (non)existence of lower scalar curvature bounds to the
  existence of certain distance-decreasing maps.  We also give a
  sufficient condition for the existence of a limiting scalar curvature
  measure in the backward limit of a Ricci flow solution.
\end{abstract}

\maketitle

\section{Introduction}

In \cite{Lott (2021)} we proved the following result, in which a lower bound on
scalar curvature gives a restriction on the existence of distance-nonincreasing
maps of nonzero degree. Let $R$ denote scalar curvature and let
$H$ denote mean curvature.

\begin{theorem} \label{1.1}
  Let $N$ and $M$ be compact connected Riemannian manifolds-with-boundary
  of the same even dimension.
  Let $f \: : \: (N, \partial N) \rightarrow
  (M, \partial M)$ be a smooth spin map and let
  $\partial f \: : \: \partial N \rightarrow \partial M$ denote the
  restriction to the boundary.
  Suppose that
\begin{itemize}
  \item $f$ is $\Lambda^2$-nonincreasing and $\partial f$ is
  distance-nonincreasing,
  \item $M$ has nonnegative curvature operator and
    $\partial M$ has nonnegative second fundamental form,
    \item 
    $R_N \ge f^* R_M$ and
      $H_{\partial N} \ge (\partial f)^* H_{\partial M}$,
      \item 
        $M$ has nonzero Euler characteristic and
        \item 
          $f$ has nonzero degree.
\end{itemize}
Then $R_N = f^* R_M$ and $H_{\partial N} = (\partial f)^* H_{\partial M}$.

Furthermore,
  \begin{itemize}
  \item    If $0 < \Ric_M < \frac12 R_M g_M$ then $f$ is a Riemannian covering
    map.
\item If $\Ric_M > 0$ and $f$ is distance-nonincreasing then $f$ is a
  Riemannian covering map.
\item If $M$ is flat then $N$ is Ricci-flat.
  \end{itemize}
\end{theorem}

In particular, the lower scalar curvature bound $R_N \ge f^* R_M$ means that
it is impossible for $f$ and $\partial f$
to be distance-decreasing (i.e. have Lipschitz constant less than one),
with $f$ having nonzero degree,
and to also have $H_{\partial N} > (\partial f)^* H_{\partial M}$.
Theorem \ref{1.1} follows earlier work by Llarull \cite{Llarull (1998)} and
Goette-Semmelmann \cite{Goette-Semmelmann (2002)}; 
we refer to \cite{Lott (2021)}
for background and generalizations.  The first main
result of the present paper is a converse
and shows that the lack of a lower bound on the
scalar curvature implies that such distance-decreasing maps do exist.

\begin{theorem} \label{1.2}
Given $n > 1$ and $K \in \R$,
let $Z$ be an $n$-dimensional Riemannian manifold
and let $z \in Z$ be a point where the
scalar curvature is $R_z < n(n-1)K$.  Then for any neighborhood $U$ of $z$,
there are
\begin{enumerate}
\item A codimension-zero compact submanifold-with-boundary $N \subset U$
containing $z$ that is diffeomorphic to a ball,
\item A codimension-zero compact submanifold-with-boundary $M$ in the
  $n$-dimensional model space
  of constant curvature $K$, diffeomorphic to a ball,
\item A smooth map $f : (N, \partial N) \rightarrow
  (M, \partial M)$ of nonzero degree
  so that $f$ and $\partial f$ are distance-decreasing,
  and the mean curvatures satisfy
  $H_{\partial N} > (\partial f)^* H_{\partial M}$, and
\item Numbers $\delta,l > 0$ so that for all $n \in \partial N$ and
  $t \in [0, l)$, one has
  $f(\exp_n (t \nu_{\partial N})) = \exp_{f(n)} ((1-\delta) t
  \nu_{\partial M})$,
  where $\nu_{\partial N}$ and $\nu_{\partial M}$ are
  the inward unit normals to $\partial N$ and $\partial M$, respectively.
\end{enumerate}

If $K \le 0$ then we can take $M$ to be strictly convex.
\end{theorem}

Together, Theorems \ref{1.1} and \ref{1.2} essentially
give a metric characterization of
lower scalar curvature bounds.
While the geometric meaning of
scalar curvature may be hard to understand, the metric
characterization is in terms of mean curvature, which is more tractable.

The proof of Theorem \ref{1.2} is by induction on $n$, as in the proof of a
related result by Gromov in \cite{Gromov (2014)}.
Item (4) in the conclusion of the theorem
is just for technical convenience, in order to simplify the induction
argument.  In the induction step, it is fairly easily to obtain
cylindrical regions that
satisfy the conclusions of the theorem,
along the lines of \cite{Gromov (2014)}, but have codimension-two
singularities.  The main technical issue is to smooth the singularities
while simultaneously
maintaining the distance-decreasing property and the inequality on
mean curvatures. Given Theorem \ref{1.2}, one can somewhat
simplify Gromov's proof of the preservation of lower scalar curvature bounds
under $C^0$-limits of smooth Riemannian metrics \cite{Gromov (2014)}.

The second main result of the paper is about the existence of a
limiting scalar curvature measure, as $t \rightarrow 0$,
for a Ricci flow coming out of a
metric space.  If there is going to be a finite limiting measure then
by looking at the total scalar curvature along the flow, one sees that
the finiteness
of $\int_0^T \int_M (R^2 - 2|\Ric|^2) \dvol_{g(t)} dt$ is a necessary
condition (equation (\ref{3.25})). The next theorem essentially says that it is
also sufficient.

\begin{theorem} \label{1.3}
  Let $(M, g(t))$, $t \in (0, T]$, be a Ricci flow solution on a compact
    $n$-dimensional manifold $M$ satisfying
\begin{enumerate}
\item    $|\Rm_{g(t)}| < \frac{A}{t}$ for some $A < \infty$ and all $t$, 
  \item $\Ric_{g(t)} \ge E g(t)$ for some $E > - \infty$ and all $t$, and
\item $R^2 - 2|\Ric|^2 \in L^1((0, T] \times M; dt \dvol_{g(t)})$.
\end{enumerate}
Then there is a limit $\lim_{t \rightarrow 0} R_{g(t)} \dvol_{g(t)} =
\mu_0$ in the weak-$\star$ topology.
  \end{theorem}

One's first approach to proving Theorem \ref{1.3}
might be to fix a test function
$f$ and consider the time evolution of $\int_M f  R_{g(t)} \dvol_{g(t)}$.
This turns out to not be useful.  Instead we let $f$ evolve by a backward
heat equation and use heat kernel estimates from
\cite{Bamler-Cabezas-Rivas-Wilking (2019)}.

Using Theorem \ref{1.3}, we show the existence of a
subsequential limiting scalar curvature
measure on a class of Ricci limit spaces.
Recall that a Riemannian manifold has
$2$-nonnegative curvature operator if at each point, the sum of the
two lowest eigenvalues of the curvature operator is nonnegative.

\begin{theorem} \label{1.4}
  Given $D, \widehat{A} < \infty$ and $v_0 > 0$, let
  $\{(M_i, g_i)\}_{i=1}^\infty$ be a sequence of compact
  $n$-dimensional Riemannian manifolds, $n \ge 4$, such that
  \begin{enumerate}
  \item $\diam(M_i, g_i) \le D$,
  \item $\vol(M_i, g_i) \ge v_0$,
  \item $(M_i, g_i)$ has $2$-nonnegative curvature operator, and
    \item $\int_{M_i} R_{g_i} \: \dvol_{g_i} \le \widehat{A}$.
  \end{enumerate}
  Then after passing to a subsequence, there is a Gromov-Hausdorff limit
  $(X_\infty, d_\infty)$ with a measure $\mu_0$, along with
  continuous Gromov-Hausdorff approximations $\eta_i : M_i \rightarrow
  X_\infty$, such that
  \begin{enumerate}
  \item $\lim_{i \rightarrow \infty} \left( \eta_i \right)_*
    \left( R_{g_i} \dvol_{g_i} \right) \stackrel{weak-\star}{=} \mu_0$, and
  \item There is a smooth $1$-parameter family of Riemannian
    metrics $\{g(t)\}_{t \in (0, T]}$ on $X_\infty$, with
      $2$-nonnegative curvature operator, so that
      $\lim_{t \rightarrow 0} (X_\infty, g(t)) \stackrel{GH}{=}
      (X_\infty, d_\infty)$ and
      $\lim_{t \rightarrow 0} R_{g(t)} \dvol_{g(t)}
      \stackrel{weak-\star}{=} \mu_0$.
  \end{enumerate}
\end{theorem}

Condition (4) in the Theorem (for some $\widehat{A}$)
may in fact follow from the other three conditions.
From the hypotheses of Theorem \ref{1.4},
the subsequential existence of a limiting metric space
(in the Gromov-Hausdorff topology) and a limiting scalar curvature measure
(in the weak-$\star$ topology) is automatic.  The content of the
theorem is that the metric space and the scalar curvature measure also
arise as a continuous limit, coming from a Ricci flow.

I thank Antoine Song for discussions and the referee for helpful comments.

\section{Proof of Theorem \ref{1.2}}

It would be unwieldy to write out equations for all the steps in the
proof of Theorem \ref{1.2}, so we give the main ingredients.
The proof is by induction on $n$, as in \cite[Section 4.9]{Gromov (2014)}.
\begin{lemma} \label{2.1}
  The theorem is true in dimension two.
\end{lemma}
\begin{proof}
If $n = 2$, choose normal coordinates
  $(r, \theta)$ around $z$ and $(r^\prime, \theta^\prime)$ around a point
  $m \in M$.  Given $\alpha \in (0,1]$, let $N$ be given by
  $r \le r_0$ and let $M$ be given by $r^\prime \le r_0^\prime = \alpha r_0$.
  In the normal coordinates, to leading order,
 \begin{align} \label{2.2}
  g_N \sim & \: dr^2 + r^2 (1 - \frac16 R_z r^2) d\theta^2, \\
  g_M \sim & \: (dr^\prime)^2 +
             (r^\prime)^2 (1 - \frac13 K (r^\prime)^2) (d\theta^\prime)^2, \notag \\
  H_{\partial N} \sim & \: \frac{1}{r_0} - \frac{1}{6} R_z r_0, \notag \\
    H_{\partial M} \sim & \: \frac{1}{r^\prime_0} - \frac{1}{3} K r_0^\prime. \notag 
    \end{align}
  Define a Lipschitz function
  $F$ by $F(r, \theta) = (\alpha r, \theta)$.
Then 
\begin{align} \label{2.3}
  F^* g_M \sim & \: \alpha^2 dr^2 + \alpha^2 r^2
                 (1 - \frac13 K \alpha^2 r^2) d\theta^2, \\
  (\partial F)^* H_{\partial M} \sim & \: \frac{1}{\alpha r_0} -
                                       \frac{1}{3} K \alpha r_0. \notag
\end{align}
For small $r_0$, if $\alpha = 1$ then $F$ is distance-nonincreasing,
$\partial F$ is distance-decreasing and
$H_{\partial N} > (\partial F)^*
  H_{\partial M}$. By continuity, if
  $\alpha$ is slightly less than one then
  $F$ and $\partial F$
  will be distance-decreasing and we will
  still have $H_{\partial N} > (\partial F)^*
  H_{\partial M}$. If $f$ is the result of slightly smoothing $F$ near $z$ then
  it will satisfy the conclusion of the
  theorem.
\end{proof}

Now suppose that $n > 2$ and the theorem is true in dimension
  $n-1$.

    \begin{lemma} \label{2.4}
    There is a
    (small) minimal
    hypersurface $V \subset U$ containing $z$ and a unit normal
    vector $v \in T_zZ$ so that
    \begin{itemize}
    \item The second fundamental form $A_z$ at $z$ has $|A_z| \ll 1$,
    \item $\Ric(v,v) < (n-1)K$, and
    \item The scalar curvature of $V$ at $z$ satisfies
      $R^\prime_z < (n-1)(n-2) K$.
     \end{itemize}
  \end{lemma}
  \begin{proof}
    Multiplying $g_Z$ by a large constant $\lambda$,
    the geometry of a unit ball around
    $z$ becomes closer and closer to Euclidean.
    Given a unit vector $w \in T_zZ$, consider the foliation of the
    rescaled ball of radius two given by hyperplanes perpendicular to
    $w$ with respect to the Euclidean metric in the normal coordinates.
    The leaves are minimal with respect to the Euclidean metric.
    Consider the part of the foliation whose height varies between
    $- 1.5$ and $1.5$.   
    Using stability results as in
    \cite{White (1987),White (1991)},
    if $\lambda$ is large enough then there is a small
    $C^k$-perturbation
    of the foliation by minimal hypersurfaces that preserves the
    intersections of the leaves with the sphere of radius two.
    This restricts to a minimal foliation of the unit ball with
    arbitrarily small
    second fundamental form, if $\lambda$ is large enough. Let
    $\alpha(w)$ be the choice of unit normal, to the leaf at $z$, that is
 close to $w$.
    For sufficiently large $\lambda$,
    the map $\alpha : S^{n-1} \rightarrow S^{n-1}$ is a local diffeomorphism,
    hence is surjective.  We will take $w = \alpha^{-1}(v)$ for an
      appropriately chosen $v$ that is specified below and let $V$ be the
      corresponding minimal leaf through $z$.

From the Gauss-Codazzi equation, the scalar curvature
$R^\prime_{{z}}$ of ${z}$ in $V$ is given by
\begin{equation} \label{2.5}
  R^\prime_{{z}} = R_{{z}} - 2 \Ric(v,v) + (\Tr(A))^2 -
  \Tr(A^2).
\end{equation}      
Put
\begin{equation} \label{2.6}
  \widehat{\Ric} = \Ric - (n-1) K g,
\end{equation}
with trace $\widehat{R} = R - n(n-1)K$.
Let $\widehat{R}_{11} \le \widehat{R}_{22} \le \ldots \le \widehat{R}_{nn}$
be the eigenvalues of $\widehat{\Ric}_z$.

If $\widehat{R}_{nn} < 0$, let
$v$ be a corresponding unit eigenvector. Then $\Ric(v,v) < (n-1)K$ and 
\begin{align} \label{2.7}
R_{{z}} - 2 \Ric(v,v) & = R_{11} + \ldots + R_{n-2,n-2} + 
                        (R_{n-1,n-1} - R_{nn}) \\
  & \le R_{11} + \ldots + R_{n-2,n-2} <
  (n-1)(n-2)K. \notag
\end{align}

  If $\widehat{R}_{nn} = 0$, let $v$ be a unit vector that is a slight
  perturbation from $\Ker(\widehat{\Ric})$.
  Then $\Ric(v,v) < (n-1)K$ and we still have
  $ R_{{z}} - 2 \Ric(v,v) < (n-1)(n-2)K$.

  If 
  $\widehat{R}_{nn} > 0$ then the quadratic form $\widehat{\Ric}$ is
  indefinite and
  for any $\delta > 0$, we can find a unit vector $v$ so that
  $- \delta < \widehat{\Ric}(v,v) < 0$.  Then
  \begin{equation} \label{2.8}
    (n-1)K - \delta < \Ric(v,v) < (n-1)K,
  \end{equation}
  so
  $R_{{z}} - 2 \Ric(v,v) < R_z - 2(n-1)K + 2 \delta$. Taking
    $\delta$ small enough, we can ensure that 
    $R_{{z}} - 2 \Ric(v,v) < (n-1)(n-2)K$.
    
    In any case, we can achieve a negative upper bound on
    $R_{{z}} - 2 \Ric(v,v) - (n-1)(n-2)K$ that is independent of
    $\lambda$.
 Finally, taking $\lambda$
 large enough to ensure that $|(\Tr(A))^2 - \Tr(A^2)|$ is small, we obtain
 from (\ref{2.5}) that $R^\prime_{{z}} < (n-1)(n-2)K$.
  \end{proof}

With reference to the $V$ of Lemma \ref{2.4},
  let $N^\prime$ be an $(n-1)$-dimensional compact
  submanifold-with-boundary of $V$ containing $z$
  obtained by applying the induction
  hypothesis, with corresponding submanifold $M^\prime$ of the
  $(n-1)$-dimensional model space
  of constant curvature $K$, and with map
  $f^\prime : N^\prime \rightarrow M^\prime$ of nonzero degree so that
  $f^\prime$ and $\partial f^\prime$ are distance-decreasing, and
  $H_{\partial N^\prime} > (\partial f^\prime)^\star H_{\partial M^\prime}$. 
  Taking $N^\prime$ small enough, we can assume that
  the unit normal vector $\nu_{N^\prime}$ satisfies $\Ric(\nu_{N^\prime},
  \nu_{N^\prime}) < (n-1)K$ on
  $N^\prime$.

  For small $\epsilon > 0$, let $N^{(2)}$ be the cylindrical
  region $\{\exp(u \nu_{N^\prime}) : |u| \le \epsilon\}$ in $Z$. Similarly, put
  $M^{(2)} = \{\exp(u \nu_{M^\prime}) : |u| \le (1-\delta_{N^\prime})
  \epsilon\}$ in the
  $n$-dimensional model space of constant curvature $K$, where
  $\nu_{M^\prime}$ is a unit normal field to $M^\prime$ and
  $\delta_{N^\prime}$ is
  the parameter appearing in the induction hypothesis.
  See Figure \ref{figure1}, which illustrates the case $K=0$.
\begin{figure}[h!]
  \includegraphics[width=3in]{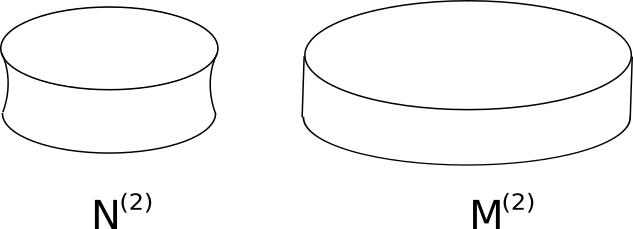}
  \caption{}
  \label{figure1}
\end{figure}
In what follows we can always reduce $l_{N^\prime}$ and $\delta_{N^\prime}$.
  Define $f^{(2)} : N^{(2)} \rightarrow M^{(2)}$ by
  $f^{(2)}(\exp_{n^\prime}(u \nu_{N^\prime})) =
  \exp_{f^{\prime}(n^\prime)}((1-\delta_{N^\prime}) u \nu_{M^\prime})$.
  Note that $\partial N^{(2)}$ has a top face and a bottom face,
both diffeomorphic to $N^\prime$, 
and an annular region diffeomorphic to $[- \epsilon, \epsilon] \times
\partial N^\prime$.  The annular region meets the top face orthogonally in a
codimension-two stratum diffeomorphic to $\partial N^\prime$, with a
similar statement for the bottom face. The maps $f^{(2)}$ and
$\partial f^{(2)}$ are distance-decreasing.

  Along the geodesics in $N^{(2)}$ normal to $N^\prime$, we have
  \begin{equation} \label{2.9}
\frac{dH}{dt} + \Tr(A^2) = - \Ric(\gamma^\prime, \gamma^\prime).        
  \end{equation}
  For small $\epsilon$, 
  if $N^\prime$ is taken small enough then
$|A_z| \ll 1$ and
on the top and bottom faces of $N^{(2)}$, we have
  $H \sim - \epsilon \Ric(\nu_{N^\prime},\nu_{N^\prime})$.
  Similarly,
  on the top and bottom faces of $M^{(2)}$, we have
    $H \sim - (n-1) \epsilon \left( 1-\delta_{N^\prime} \right)K$.
  Including the annular region over $\partial N^\prime$, if $\epsilon$
  and $\delta_{N^\prime}$ are
    small then
    $H_{\partial N^{(2)}} >
    (\partial f^{(2)})^\star H_{\partial M^{(2)}}$.

    We can assume that the parameter $l_{N^\prime}$ in the induction
    assumption is less than
  the focal radius of $\partial N^\prime \subset N^\prime$, and $l_{N^\prime} \ll \epsilon$.
Given $\epsilon^{\prime} \ll l_{N^\prime}$, let
    $N^{(3)}$ be the points inside of
    $N^{(2)}$ that have distance at least $\epsilon^{\prime}$
    from $\partial N^{(2)}$.  Let
    $N^{(4)}$ be the
    $\epsilon^{\prime}$-neighborhood of
    $N^{(3)}$ in $Z$.
    Do a similar construction for $M$, to obtain 
    $M^{(4)}$.

    The boundary $\partial N^{(4)}$ is $C^{1,1}$-regular and
    has a decomposition into a top face $F_+^{(4)}$, a bottom face
    $F_-^{(4)}$, an
    annular belt $A^{(4)}$, an upper tube $T_+^{(4)}$ and a lower tube
    $T_-^{(4)}$. See Figure \ref{figure2}.
    There is a similar decomposition of $\partial M^{(4)}$.
\begin{figure}[h!]
  \includegraphics[width=3in]{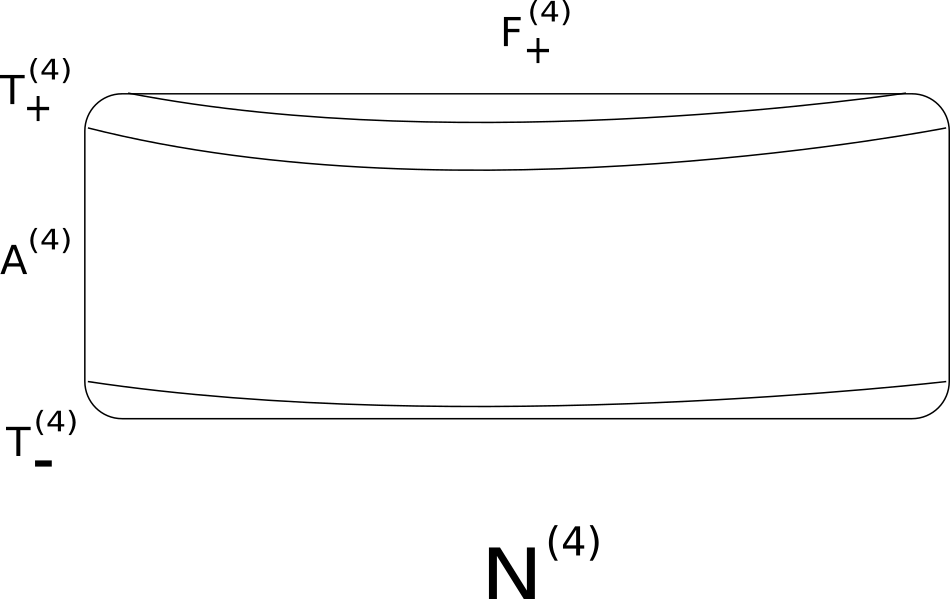}
  \caption{}
  \label{figure2}
  \end{figure}

  Consider the top tubular region $T_+^{(4)}$. It has two
  boundary components
  $\partial T_{+,1}^{(4)}$ and $\partial T_{+,2}^{(4)}$, with
  $\partial T_{+,1}^{(4)}$ also being a boundary
  component of $F_+^{(4)}$ and $\partial T_{+,2}^{(4)}$ also being a boundary
  component of $A^{(4)}$.

Let $p : N^{(2)} \rightarrow N^\prime$ be projection onto the second factor in
the diffeomorphism $N^{(2)} \cong [ - \epsilon, \epsilon] \times
N^\prime$.
Let $\widehat{p} : M^{(2)} \rightarrow M^\prime$ be the analogous map
on $M^{(2)}$.
Given $n^\prime \in \partial N^\prime$,
put $L_{n^\prime} = \{\exp_{n^\prime} (t \nu_{\partial N^\prime}) :
0 \le t \le l_{N^\prime} \} \subset N^\prime$ and put
$G_{n^\prime} = p^{-1}(L_{n^\prime})$.
  Put  $F_{\pm,n^\prime}^{(4)} = F_{\pm}^{(4)} \cap G_{n^\prime}$,
  $T_{\pm,n^\prime}^{(4)} = T_{\pm}^{(4)} \cap G_{n^\prime}$ and
  $A_{n^\prime}^{(4)} = A^{(4)} \cap G_{n^\prime}$.
  See the left-hand picture in Figure \ref{figure3}, which
  illustrates $G_{n^\prime} \cap N^{(4)}$.
\begin{figure}[h!]
  \includegraphics[width=2in]{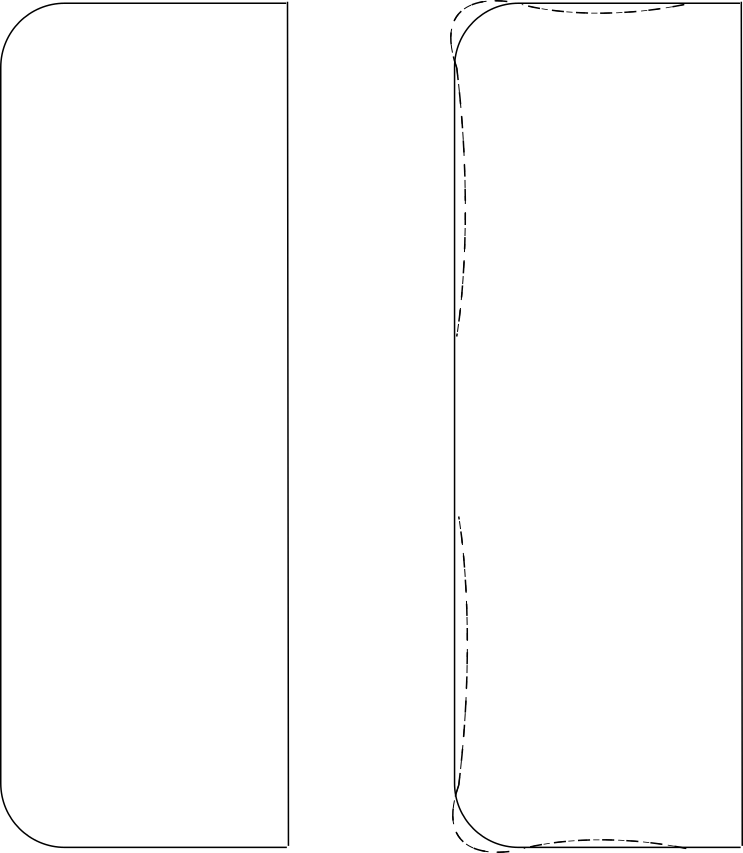}
  \caption{}
  \label{figure3}
  \end{figure}  Define
$\widehat{G}_{m^\prime}$,
  $\widehat{F}_{\pm,m^\prime}^{(4)}, \widehat{T}_{\pm,m^\prime}^{(4)}$ and
  $\widehat{A}_{m^\prime}^{(4)}$ similarly for $M^{(4)}$.

We define a map $\partial f^{(4)} : \partial N^{(4)} \rightarrow
\partial M^{(4)}$ as follows. Given $n^\prime \in \partial N^\prime$, we first
send the curve $T_{+,n^\prime}^{(4)} \cup A_{n^\prime}^{(4)} \cup
T_{-,n^\prime}^{(4)}$ to
$\widehat{T}_{+,(\partial f^\prime)(n^\prime)}^{(4)} \cup
\widehat{A}_{(\partial f^\prime)(n^\prime)}^{(4)} \cup
\widehat{T}_{-,(\partial f^\prime)(n^\prime)}^{(4)}$
piecewise linearly with respect to arc length.
Next, given the point $x = \partial T_{+,1}^{(4)} \cap G_{n^\prime}$,
write $p(x)$ as $\exp_{n^\prime}(\tau_{n^\prime} \nu_{\partial N^\prime})$ for some
$\tau_{n^\prime} \in (0, l_{N^\prime})$. The parameter $\tau_{n^\prime}$ is
comparable to $\epsilon^\prime$.
Define
$\widehat{\tau}_{(\partial f^\prime)(n^\prime)}$ similarly for $M^\prime$.
Let $\lambda_{n^\prime}$ be the increasing linear bijection from
$[\tau_{n^\prime}, l_{N^\prime}]$ to
$[\widehat{\tau}_{(\partial f^\prime)(n^\prime)}, 
(1 - \delta_{N^\prime}) l_{N^\prime}]$.
Given $t \in [\tau_{n^\prime}, l_{N^\prime}]$, let
$y_t$ be the point in $F_{+,n^\prime}^{(4)}$ with $p(y_t) =
\exp_{n^\prime}(t \nu_{\partial N^\prime})$ and let
$(\partial f^{(4)})(y_t)$ be the point in
$\widehat{F}_{+,(\partial f^\prime)(n^\prime)}^{(4)}$ whose image under
$\widehat{p}$ is
$\exp_{(\partial f^\prime)(n^\prime)}(\lambda_{n^\prime}(t) \nu_{\partial M^\prime})$.
Define $\partial f^{(4)}$ on the remaining points of
$F_{+,n^\prime}^{(4)}$ to be the same as 
$\partial f^{(2)}$. Finally, define $\partial f^{(4)}$ on
$F_{-,n^\prime}^{(4)}$ by a similar construction.

The lengths of $T_{\pm,n^\prime}^{(4)}$ and  
$\widehat{T}_{\pm,m^\prime}^{(4)}$ are $\frac{\pi}{2\epsilon^\prime} +
O((\epsilon^\prime)^0)$
as $\epsilon^\prime \rightarrow \infty$.  The Lipschitz constant of
$\partial f^{(4)}$ is $1 + O(\epsilon^\prime)$. We can extend 
$\partial f^{(4)}$ to a map
$f^{(4)} : N^{(4)} \rightarrow M^{(4)}$, sending
$G_{n^\prime} \cap N^{(4)}$ to $\widehat{G}_{(\partial f^\prime)(n^\prime)}
\cap M^{(4)}$, whose Lipschitz constant is $1 + O(\epsilon^\prime)$.

By tube formulas, the mean curvature on $T_\pm^{(4)}$
is $\frac{1}{\epsilon^\prime} + O(\epsilon^\prime)$, and
similarly for the tube regions of $\partial M^{(4)}$
\cite[Theorem 9.23]{Gray (2004)}.
We now perturb
$N^{(4)}$ to increase the mean curvature on $T_\pm^{(4)}$.
To do so, we effectively borrow some of the mean curvature from
$F_\pm^{(4)}$ and $A^{(4)}$

We do some preliminary calculations.
Let $\phi : [0, \infty) \rightarrow \R$ be a smooth nonnegative function
such that $\phi(0) = 0$, $\phi^\prime(0) = 1$ and $\phi(x)$ vanishes when
$x \ge 2$. 
Given constants $c, c^\prime > 0$ so that $c l_{N^\prime} \gg 1$ and
$c \epsilon^\prime \ll 1$, and 
$L < \infty$, 
define $u_L \in C^{1,1}(\R)$ by
\begin{equation} \label{2.10}
  u_L =
  \begin{cases}
\frac{c^\prime L}{2c} \phi(-cx) & x < 0 \\
    \frac{c^\prime}{2} x \left( x - L \right) & 0 \le x \le
    L \\
    \frac{c^\prime L}{2c} 
    \phi(c(x-L)) & x > L.
  \end{cases}
  \end{equation}
  Then
  \begin{equation} \label{2.11}
  u_L^{\prime \prime} =
  \begin{cases}
\frac{c^\prime c L}{2} \phi^{\prime \prime}(-cx) & x < 0 \\
    c^\prime & 0 < x <
    L \\
    \frac{c^\prime c L}{2}
    \phi^{\prime \prime}(c(x - L)) 
& x > L.
  \end{cases}
  \end{equation}

  Let $d_1$ be the intrinsic distance function
  on $\partial N^{(4)}$ from $\partial T_{+,1}^{(4)}$ and let 
  $d_2$ be the intrinsic distance function
  on $\partial N^{(4)}$ from $\partial T_{+,2}^{(4)}$. Let
  $\pi_1 : \partial N^{(4)} \rightarrow \partial T_{+,1}^{(4)}$ be
  nearest point projection with respect to the intrinsic distance on
  $\partial N^{(4)}$, and similarly for $\pi_2$.  (In the application, the
  nearest point will be unique.)
  Define a function
  $V_+ \in C^{1,1}(\partial N^{(4)})$ by
  \begin{equation} \label{2.12}
 V_+  =
  \begin{cases}
- \: \frac{c^\prime }{2} d_1 d_2 & \text{on } T_+^{(4)}  \\
    \frac{c^\prime}{2c} (d_2 \circ \pi_1) \: \phi(c d_1) & \text{on } F_+^{(4)}
    \\
    \frac{c^\prime}{2c} (d_1 \circ \pi_2) \: \phi(c d_2) & \text{on } A_+^{(4)} \\
    0 & \text{on } T_-^{(4)} \cup F_-^{(4)}.
  \end{cases}
  \end{equation}
  Define  $V_- \in C^{1,1}(\partial N^{(4)})$ similarly, replacing
  $T_+^{(4)}$ by $T_-^{(4)}$. Put $V = V_+ + V_-$.
Deform $\partial N^{(4)}$ by distance $V$ in the
inward normal direction. See the right-hand picture in Figure \ref{figure3},
where
the deformation is indicated by dashed lines.
Let $N^{(5)}$ be the region bounded by the
ensuing hypersurface, i.e. $\partial N^{(5)}$ is the image of the
$C^{1,1}$-diffeomorphism $D : \partial N^{(4)} \rightarrow \partial N^{(5)}$
given by
$D(x) = \exp_x(V(x) \nu_{\partial N^{(4)}})$.
Note that the deformation is outward on $T^{(4)}_\pm$, of magnitude
comparable to $c^\prime (\epsilon^\prime)^2$, and inward on the rest of
$\partial N^{(4)}$, of magnitude comparable to 
$c^\prime c^{-1} \epsilon^{\prime}$.
 
The first variation formula for mean curvature is
\begin{equation} \label{2.13}
H^\prime = \triangle V + (|A|^2 + \Ric(\mu, \mu))) V.
\end{equation}
If $x$ denotes the length variable on a minimal arc in $T_+^{(4)}$
between $q \in \partial T_{+,2}^{(4)}$ and $\pi_1(q) \in \partial T_{+,1}^{(4)}$
then on $T_+^{(4)}$, to leading order $\triangle V \sim \frac{d^2}{dx^2} V$ and
$|A|^2 \sim (\epsilon^{\prime})^{-2}$. From (\ref{2.11}), we deduce that
on $T_\pm^{(4)}$
the change in $H$ roughly
ranges between $c^\prime \left( 1 - \frac{\pi^2}{32} \right)$ and
$c^\prime$. On the rest of $\partial N^{(4)}$, the change in $H$ is
bounded in magnitude by $\const c^\prime \epsilon^{\prime} (c+c^{-1})$.
Put $M^{(5)} = M^{(4)}$.

Define subsets of $\partial N^{(5)}$ by
$\widetilde{F}_{\pm}^{(5)} =
D({F}_{\pm}^{(4)})$,
$\widetilde{T}_{\pm}^{(5)} = D({T}_{\pm}^{(4)})$ and
$\widetilde{A}^{(5)} = D(A^{(4)})$.
Define $\partial f^{(5)} : \partial N^{(5)} \rightarrow \partial M^{(5)}$
by $\partial f^{(5)} = (\partial f^{(4)}) \circ D^{-1}$.
On $\widetilde{T}_{\pm}^{(5)}$,
the map $\partial f^{(5)}$ has a Lipschitz bound comparable to that of
$\partial f^{(4)}$, namely $1 + O( \epsilon^{\prime})$, using the fact that
the perturbation on ${T}_{\pm}^{(4)}$ is outward. (It may seem paradoxical
that $\widetilde{T}_{\pm}^{(5)}$ lies outside of
${T}_{\pm}^{(4)}$ but has a higher mean curvature.  One way to
understand this is by looking at Figure
\ref{figure3} and comparing the total turning angle of
${T}_{+,n^\prime}^{(4)}$ with the total turning angle of the
corresponding dotted segment on the right.)
If the Lipschitz bound of $\partial f^{(2)}$ is
$1 - \sigma$, where $\sigma > 0$, then the Lipschitz bound
of $\partial f^{(5)}$
on the rest of
$\partial N^{(5)}$ is
$1 - \sigma + \const c^\prime \epsilon^\prime$. In sum,
$\partial f^{(5)}$ has a Lipschitz bound that is
$1+ O(\epsilon^{\prime})$.
We can extend $\partial f^{(5)}$ to a map
$f^{(5)} : N^{(5)} \rightarrow M^{(5)}$ which also has a Lipschitz bound
that is
$1+ O(\epsilon^{\prime})$. We can assume that
$f^{(5)}$ maps normal geodesics to normal geodesics,
in a small neighborhood of $\partial N^{(5)}$.

On $\widetilde{T}_{\pm,n^\prime}^{(5)}$,
the ratio
$\frac{H_{\partial N^{(5)}}}{(\partial f^{(5)})^\star H_{\partial M^{(5)}}}$
is bounded below by $(1 + \const \epsilon^{\prime} c^\prime)$.
If $c^\prime \epsilon^{\prime} (c+c^{-1}) \ll 1$ then
on the rest of $\partial N^{(6)}$, there is a uniform lower bound
for the ratio that
is greater than one, coming from $\partial N^{(2)}$.
Hence by taking $c^\prime$ sufficiently large and then
$\epsilon^{\prime}$ sufficiently small, we
can ensure that
the Lipschitz constant of $f^{(5)}$ is strictly less than the minimum of 
$\frac{H_{\partial N^{(5)}}}{(\partial f^{(5)})^\star H_{\partial M^{(5)}}}$.

Put $N^{(6)} = N^{(5)}$.  
Taking normal coordinates around $f^{(5)}(z) \in M^{(5)}$, let
$M^{(6)}$ be the result of a slight radial shrinking of
$M^{(5)}$. If $f^{(6)} : N^{(6)} \rightarrow M^{(6)}$ is the composite map then
we can ensure that $f^{(6)}$ and $\partial f^{(6)}$ are distance-decreasing,
while $H_{\partial N^{(6)}} >
(\partial f^{(6)})^\star H_{\partial M^{(6)}}$.

We run the mean curvature flow on 
$\partial N^{(6)}$ and $\partial M^{(6)}$ for a time $\tau \ll (\epsilon^{\prime})^2$ (c.f. \cite{Ecker-Huisken (1991)})
to obtain smooth hypersurfaces
$\partial N$ and $\partial M$, and hence $N$ and $M$.
As the mean curvature obeys a diffusion-type equation under mean curvature flow,
the main effect on the mean curvature at a point will be to average the
mean curvature with respect to 
a Gaussian centered at the point with scale on the order of
$\sqrt{\tau}$. In particular, if $\tau$ is small enough then
for the map $\partial f : \partial N \rightarrow
\partial M$, obtained from $f^{(6)}$ by 
following the flows,
the inequalities will be preserved. We can then 
extend it to a smooth distance-decreasing map
$f : N \rightarrow M$. By construction, after choosing orientations on
$N$ and $M$ (which are diffeomorphic to balls), the degree is nonzero.
Finally, after a small perturbation of $f$, we can assume that
there are numbers $\delta,l > 0$ so that for all $n \in \partial N$ and
  $t \in [0, l)$, one has
  $f(\exp_n t \nu_{\partial N}) = \exp_{f(n)} ((1-\delta) t
  \nu_{\partial M})$.

  If $K < 0$ then for sufficiently small $\tau$,
  the preceding steps preserve the
  strict convexity of $M$.
  If $K = 0$ then they preserve the convexity of $M$ and we can
  slightly perturb $M$ at the end, for example by the mean curvature flow,
  to make it strictly convex.

  \section{Proof of Theorem \ref{1.3}}

\begin{lemma} \label{3.1} There is a constant $C^\prime = C^\prime(n,A) < \infty$ so that
if $0 < s \le t$ then
  \begin{equation} \label{3.2}
    d_s - C^\prime \left( \sqrt{t} - \sqrt{s} \right) \le d_t \le
    e^{E(s-t)} d_s.
    \end{equation}  
\end{lemma}
\begin{proof}
  This follows from distance distortion estimates for Ricci flow, as in
  \cite[Remark
27.5 and Corollary 27.16]{Kleiner-Lott (2008)}.
  \end{proof}

\begin{corollary} \label{3.3}
The diameter of $(M, g(t))$ is uniformly bounded above in $t$.
  \end{corollary}

\begin{lemma} \label{3.4}
There are some $v_0, A^\prime > 0$ so that
for all $(x,t) \in M \times (0, T]$,
  \begin{enumerate}
\item   $\vol_{g(t)} \left( B_{g(t)}(x, 1) \right) \ge v_0$, and
\item    $\vol_{g(t)} \left( B_{g(t)}(x, \sqrt{t}) \right) > A^\prime
  t^{n/2}$
  \end{enumerate}
      \end{lemma}
\begin{proof}
  From the evolution of volume under Ricci flow,
  \begin{equation} \label{3.5}
    \frac{d}{dt} \Vol(M, g(t)) =
    - \int_M R_{g(t)} \: \dvol_{g(t)} \le - nE \Vol(M, g(t)). 
  \end{equation}
  It follows that $\Vol(M, g(t)) \ge e^{nE(T-t)} \Vol(M, g(T))$.
  The lower
  Ricci curvature bound, the diameter bound from Corollary \ref{3.3}
  and Bishop-Gromov
  comparison now give
  numbers $v_0, A^\prime$ as in the statement of the lemma.
\end{proof}

\begin{lemma} \label{3.6}
  Let $F$ be a solution of the backward heat equation
  \begin{equation} \label{3.7}
\partial_t F = - \triangle F.
  \end{equation}
Then $\max |F|$ and 
$e^{-2Et} \max |\nabla F|$ are nondecreasing in $t$.
\end{lemma}
\begin{proof}
  From the maximum principle,
  $\max |F|$ is nondecreasing in $t$.
  Next, we have
  \begin{equation} \label{3.8}
    \partial_t |\nabla F|^2 = 2 \Ric(\nabla F, \nabla F) - 2
    \langle \nabla F, \nabla \triangle F \rangle
  \end{equation}
  and
    \begin{equation} \label{3.9}
      \triangle |\nabla F|^2 = 2 \langle \nabla F, \nabla \triangle F \rangle
      + 2 |\Hess F|^2 + 2 \Ric(\nabla F, \nabla F). 
    \end{equation}
    Hence
      \begin{equation} \label{3.10}
        (\partial_t + \triangle) |\nabla F|^2 \ge 4 \Ric(\nabla F, \nabla F)
        \ge  4 E |\nabla F|^2,
      \end{equation}
      or
            \begin{equation} \label{3.11}
              (\partial_t + \triangle) \left( e^{-4Et} |\nabla F|^2 \right)
              \ge 0.
      \end{equation}
            By the maximum principle, $e^{-4Et}
            \max |\nabla F|^2$ is nondecreasing in $t$.
This proves the lemma.
\end{proof}

\begin{lemma} \label{3.12}
Given $f \in C^\infty(M)$,
there is a function $\alpha : (0, T] \rightarrow \R^+$ with
  $\lim_{t \rightarrow 0} \alpha(t) = 0$ having the
  following property.
Given $\widehat{t} \in (0, T]$,
  let $F$ be the solution to (\ref{3.7}) on $(0, \widehat{t}]$
    with $F(\widehat{t}) = f$. 
If $s \in (0, \widehat{t}/2]$ then $\| F(s) - f \|_\infty
  \le \alpha(\widehat{t})$.
\end{lemma}
\begin{proof}
  Let $G(x,t;y,s)$, $0 < s < t$,
  be the Green's function for (\ref{3.7}), meaning that
  for fixed $(x,t)$, the function $G(x,t;\cdot, \cdot)$ satisfies
  \begin{equation} \label{3.13}
(\partial_s + \triangle_{y,s})G(x,t;\cdot, \cdot) = 0,
  \end{equation}
  and $\lim_{s \rightarrow t} G(x,t;y,s) = \delta_x(y)$.
  Then $G$ is positive and for given $(y,s)$, one has
  \begin{equation} \label{3.14}
    \int_M G(x,t;y,s) \: \dvol_{g(t)}(x) = 1.
  \end{equation}
  Also,
  \begin{equation} \label{3.15}
    F(y,s) = \int_M G(x,\widehat{t};y,s) \: f(x) \: \dvol_{g(\widehat{t})}(x).
    \end{equation}
  
  By \cite[Proposition 3.1]{Bamler-Cabezas-Rivas-Wilking (2019)},
  there is a constant $C = C(n,A) < \infty$
  so that
  \begin{equation} \label{3.16}
G(x,t;y,s) < C t^{- \: \frac{n}{2}} e^{- \: \frac{d_s^2(x,y)}{Ct}}
  \end{equation}
  whenever $s \le \frac{t}{2}$.

  Given $L < \infty$, 
we have
    \begin{align} \label{3.17}
      F(y,s) - f(y) = &
      \int_M G(x,\widehat{t};y,s) \: (f(x) - f(y)) \:
      \dvol_{g(\widehat{t})}(x) \\
      = & \int_{B_{g(\widehat{t})}(y, L \sqrt{t})} G(x,\widehat{t};y,s) \: (f(x) - f(y)) \:
      \dvol_{g(\widehat{t})}(x) + \notag \\
& \int_{M-B_{g(\widehat{t})}(y, L \sqrt{t})} G(x,\widehat{t};y,s) \: (f(x) - f(y)) \:
      \dvol_{g(\widehat{t})}(x), \notag
    \end{align}
    so
    \begin{align} \label{3.18}
      |F(y,s) - f(y)| \le &
      \max_{x \in B_{g(\widehat{t})}(y, L \sqrt{t})} |f(x) - f(y)| + \\
& 2 (\max |f|) 
      \int_{M-B_{g(\widehat{t})}(y, L \sqrt{t})} C
      \widehat{t}^{- \: \frac{n}{2}} e^{- \: \frac{d_s^2(x,y)}{C\widehat{t}}}
      \:
      \dvol_{g(\widehat{t})}(x). \notag
      \end{align}
    From Lemma \ref{3.6},
    \begin{equation} \label{3.19}
      \max_{x \in B_{g(\widehat{t})}(y, L \sqrt{\widehat{t}})} |f(x) - f(y)|
      \le
      L \sqrt{\widehat{t}} \max |\nabla f|_{g(\widehat{t})} \le
      L \sqrt{\widehat{t}}
      e^{2E(\widehat{t}-T)} \max |\nabla f|_{g(T)}.
    \end{equation}
    From (\ref{3.2}), we have
    $d_s^2(x,y) \ge e^{2E(\widehat{t}-s)} d_{\widehat{t}}^2(x,y)
    \ge e^{-2|E|\widehat{t}} d_{\widehat{t}}^2(x,y)$. Then
    Bishop-Gromov comparison gives
    \begin{align} \label{3.20}
&      \int_{M-B_{g(\widehat{t})}(y, L \sqrt{t})} C
      \widehat{t}^{- \: \frac{n}{2}} e^{- \: \frac{d_s^2(x,y)}{C\widehat{t}}}
      \:
      \dvol_{g(\widehat{t})}(x) \le \\
&      C
      \widehat{t}^{- \: \frac{n}{2}} \vol(S^{n-1})
      \int_{L\sqrt{t}}^\infty
      e^{- \: \frac{e^{-2|E|\widehat{t}} r^2}{C \widehat{t}}} \left( \frac{1}{\sqrt{|E|}}
      \sinh(r \sqrt{|E|}) \right)^{n-1} \: dr = \notag \\
      &  C
      \vol(S^{n-1})
      \int_{L}^\infty
      e^{- \: \frac{e^{-2|E|\widehat{t}} u^2}{C}}
      \left( \frac{1}{\sqrt{|E|\widehat{t}}}
      \sinh(u \sqrt{|E| \widehat{t}} \right)^{n-1} \: du \le \notag \\
      &  C
      \vol(S^{n-1})
      \int_{L}^\infty
      e^{- \: \frac{u^2}{2C}}
      \left( 
      \sinh(u) \right)^{n-1} \: du \notag
    \end{align}
for $\widehat{t}$ small.
Taking $L = \widehat{t}^{- \: \frac14}$, the lemma follows
from (\ref{3.18}), (\ref{3.19}) and (\ref{3.20}).
\end{proof}

\begin{lemma} \label{3.21}
  If $F$ is a solution of (\ref{3.7}) then
  \begin{equation} \label{3.22}
    \frac{d}{dt} \int_M F R \: \dvol =
 -   \int_M F \left( R^2 - 2 |\Ric|^2 \right) \: \dvol.
    \end{equation}
  \end{lemma}
\begin{proof}
  We have
  \begin{align} \label{3.23}
    \frac{d}{dt} \int_M F R \: \dvol = &
    \int_M \left( \frac{\partial F}{\partial t} R \dvol +
    F \frac{\partial R}{\partial t} \dvol +
    F R \frac{d \dvol}{dt} \right)  \\
    = & 
    \int_M \left( - (\triangle F) R \dvol +
    F (\triangle R + 2 |\Ric|^2) \dvol -
    F R^2 \dvol \right) \notag \\
    = & -   \int_M F \left( R^2 - 2 |\Ric|^2 \right) \: \dvol. \notag
  \end{align}
  This proves the lemma.
  \end{proof}

\begin{lemma} \label{3.24}
  There is a $C^{\prime \prime} < \infty$ such that
  $\| R_{g(t)} \|_{L^1} \le C^{\prime \prime}$ for all $t \in (0, T]$.
\end{lemma}
\begin{proof}
  Taking $F = 1$ in (\ref{3.22}) gives
  \begin{equation} \label{3.25}
    \int_M R_{g(T)} \dvol_{g(T)} -
    \int_M R_{g(\widetilde{t})} \dvol_{g(\widetilde{t})} =
    - \int_{\widetilde{t}}^T \int_M \left( R^2 - 2 |\Ric|^2 \right)(x,t)
    \dvol_{g(t)}(x) dt.
  \end{equation}
  As $\Ric_{g(\widetilde{t})} \ge E g(\widetilde{t})$, it follows that
  $R_{g(\widetilde{t})} \ge nE$. The lemma now follows from (\ref{3.25}).  
  \end{proof}

With the hypotheses of Lemma \ref{3.12}, from (\ref{3.22}) we obtain
\begin{equation} \label{3.26}
  \int_M f R_{g(\widehat{t})} \dvol_{g(\widehat{t})} -
  \int_M F(s) R_{g(s)} \dvol_{g(s)} =
  - \int_{s}^{\widehat{t}} F(t) \left( R^2 - 2 |\Ric|^2 \right)
  \dvol_{g(t)} dt.
  \end{equation}
Now
\begin{align} \label{3.27}
& \left| \int_M f R_{g(\widehat{t})} \dvol_{g(\widehat{t})} -
    \int_M f R_{g(s)} \dvol_{g(s)} \right| \le \\ 
&      \left| \int_M f R_{g(\widehat{t})} \dvol_{g(\widehat{t})} -
      \int_M F(s) R_{g(s)} \dvol_{g(s)} \right| + \notag \\
     & \left| \int_M (F(s)-f) R_{g(s)} \dvol_{g(s)} \right|. \notag
\end{align}
From Lemma \ref{3.6} and (\ref{3.26}),
\begin{equation} \label{3.28}
  \left| \int_M f R_{g(\widehat{t})} \dvol_{g(\widehat{t})} -
  \int_M F(s) R_{g(s)} \dvol_{g(s)} \right| \le
  (\max |f|) \int_0^{\widehat{t}} \left| R^2 - 2 |\Ric|^2 \right|
  \dvol_{g(t)} dt.
\end{equation}
From Lemmas \ref{3.12} and \ref{3.24},
\begin{equation} \label{3.29}
  \left| \int_M (F(s)-f) R_{g(s)} \dvol_{g(s)} \right| \le
  C^{\prime \prime} \alpha(\widehat{t}).
  \end{equation}
Hence
\begin{align} \label{3.30}
&  \left| \int_M f R_{g(\widehat{t})} \dvol_{g(\widehat{t})} -
  \int_M f R_{g(s)} \dvol_{g(s)} \right| \le \\
&  (\max |f|) \int_0^{\widehat{t}} \left| R^2 - 2 |\Ric|^2 \right|
  \dvol_{g(t)} dt + C^{\prime \prime} \alpha(\widehat{t}). \notag
\end{align}

From (\ref{3.30}), the sequence
$\{\int_M f R_{g(2^{-j} T)} \dvol_{g(2^{-j} T)} \}_{j=0}^\infty$ is
a Cauchy sequence and so has a limit $M_f \in \R$.
Then from (\ref{3.30}),
$\lim_{t \rightarrow 0}
\int_M f R_{g(t)} \dvol_{g(t)} = M_f$.

Given $f, f^\prime \in C^\infty(M)$,
we have
\begin{equation} \label{3.31}
\left| \int_M f R_{g(t)} \dvol_{g(t)} -
\int_M f^\prime R_{g(t)} \dvol_{g(t)} \right| \le
C^{\prime \prime} \| f - f^\prime \|_\infty. 
\end{equation}
It follows that the map $f \rightarrow M_f$ extends to a
bounded linear function on $C(M)$, and
so defines a Borel measure $\mu_0$ on $M$.

This proves Theorem \ref{1.3}.

\begin{example} \label{3.32}
In dimension two, $R^2 - 2 |\Ric|^2 = 0$. 
  Let $\Sigma$ be a compact boundaryless two dimensional Alexandrov space.
  It is known that there is a Ricci flow solution $(M, g(t))$, defined for
  an interval $(0, T]$, that satisfies the assumptions of Theorem \ref{1.3}
    and for which $\lim_{t \rightarrow 0} (M, g(t)) \stackrel{GH}{=}
    \Sigma$
    \cite{Richard (2018)}.
    Theorem \ref{1.3} reproduces the canonical curvature measure
    on $\Sigma$, as defined in \cite{Reshetnyak (1993)}.
\end{example}

\begin{remark}
  The scalar curvature measure $\mu_0$ is defined using the Ricci flow.
  We do not know if it just depends on the Gromov-Hausdorff limit
  $\lim_{t \rightarrow 0} (M, g(t))$. There are examples of
  distinct Ricci flows coming out of a cone \cite{Angenent-Knopf (2019)};
  however those Ricci flows do not have a lower bound on the
  Ricci curvature.

  If the time slices $(M, g(t))$ have
  nonnegative curvature operator then by the
  uniqueness result of \cite{Lebedeva-Petrunin (2019)}, the scalar
  curvature measure
  $\mu_0$ agrees with the measure constructed there, and hence only
  depends on the limit space.
  \end{remark}
\section{Proof of Theorem \ref{1.4}}

Note that if $(M,g)$ has $2$-nonnegative curvature operator then it has
nonnegative Ricci curvature.

The next lemma is probably well known;  we give the direct proof.

\begin{lemma} \label{4.1}
  Let $H$ be a finite dimensional real inner product space and let $S$ be a
  symmetric operator on $H$. Let $\lambda_1 \le \lambda_2 \le \ldots$ be the
  eigenvalues of $S$, listed with multiplicity.  If $J$ is a
  $j$-dimensional subspace of $H$, let $P_J$ be orthogonal projection onto $J$.
  Then
  \begin{equation} \label{4.2}
    \Tr \left( P_J S P_J \right) \ge
    \sum_{i=1}^j \lambda_i.
    \end{equation}
  \end{lemma}
\begin{proof}
  By continuity and the compactness of the Grassmannian of $j$-planes in $H$,
  there is some $J$ that minimizes the left-hand side of
  (\ref{4.2}). Suppose that $J$ is a minimizer.  If $U \in O(H)$
  then $P_{UJ} = UP_JU^{-1}$, so
  $\Tr \left( S P_J \right) \le \Tr \left( S UP_JU^{-1} \right)$.
  Considering one-parameter subgroups of $O(H)$, it follows that
  \begin{equation} \label{4.3}
    0 = \Tr \left( S [\eta, P_J] \right) =
    \Tr \left( \eta [P_J, S] \right)
  \end{equation}
  for all skew-symmetric $\eta$.  As $[P_J, S]$ is skew-symmetric, it
  must vanish. If
  $H = \bigoplus_k H_k$ is the spectral decomposition of $S$ into
  eigenspaces of distinct eigenvalue then
  we must have $J = \bigoplus_k W_k$, where $W_k \subset H_k$.
  Considering $J$'s just of this form and minimizing
  $\Tr \left( P_J S P_J \right)$, the lemma follows.
  \end{proof}

\begin{lemma} \label{4.4}
  Suppose that
  \begin{enumerate}
\item    $n = 3$ and $(M,g)$ has nonnegative sectional curvature, or
\item    $n \ge 4$ and $(M,g)$ has 2-nonnegative curvature operator.
  \end{enumerate}
  Then $R^2 - 2 |\Ric|^2 \ge 0$
\end{lemma}
\begin{proof}
  We have
  \begin{equation} \label{4.5}
    R^2 - 2 |\Ric|^2 = \Tr \left( \Ric (Rg-2\Ric) \right).
  \end{equation}
  Let $\{e_i\}$ be an orthonormal basis of $T_mM$ that diagonalizes
  $\Ric$.  Then
  \begin{equation} \label{4.6}
    R^2 - 2 |\Ric|^2 = \sum_i R_{ii} (Rg-2\Ric)_{ii}. 
  \end{equation}
  Now
  \begin{equation} \label{4.7}
    (Rg-2\Ric)_{ii} = R - 2 \Ric_{ii} =
    \sum_{j,k} R_{jkjk} - 2 \sum_{j} R_{ijij} =
    \sum_{j,k \neq i} R_{jkjk}.
    \end{equation}
  If $(M,g)$ has nonnegative sectional curvature then
  the last term in (\ref{4.7}) is clearly nonnegative.  Suppose
  that $(M,g)$ has
  $2$-nonnegative curvature operator. The last term in (\ref{4.7}) is
  \begin{equation} \label{4.8}
    \sum_{j,k \neq i} R_{jkjk} = 2 \sum_{\substack{j,k \neq i \\ j < k}}
    \langle e_j \wedge e_k, {\mathcal R}(e_j \wedge e_k) \rangle.
  \end{equation}
  We apply Lemma \ref{4.1} with $H = \Lambda^2(T_mM)$ and
  $J = \spann \{ e_j \wedge e_k\}_{\substack{j,k \neq i \\ j < k}}$.
    We conclude that
    \begin{equation} \label{4.9}
      \sum_{\substack{j,k \neq i \\ j < k}}
      \langle e_j \wedge e_k, {\mathcal R}(e_j \wedge e_k) \rangle \ge
      \sum_{l=1}^{\binom{n-1}{2}} \lambda_l.
    \end{equation}
    If $n \ge 4$ then $\binom{n-1}{2} \ge 2$ and
    $\sum_{l=1}^{\binom{n-1}{2}} \lambda_l \ge
    \lambda_1 + \lambda_2 \ge 0$.  This proves the lemma.
  \end{proof}

\begin{corollary} \label{4.10}
  Let $(M, g(t))$, $t \in (0, T]$, be a Ricci flow solution on a compact
    $n$-dimensional manifold $M$.
    Suppose that
    \begin{enumerate}
    \item  $|\Rm_{g(t)}| < \frac{A}{t}$,
\item $n = 3$ and each $(M, g(t))$ has
    nonnegative sectional curvature, or $n \ge 4$ and each $(M, g(t))$ has
    $2$-nonnegative curvature operator, and
  \item There is a uniform upper bound on $\int_M R_{g(t)} \dvol_{g(t)}$.
    \end{enumerate}
Then there is a limit $\lim_{t \rightarrow 0} R_{g(t)} \dvol_{g(t)} =
\mu_0$ in the weak-$\star$ topology.
\end{corollary}
\begin{proof}
  From Lemma \ref{4.4} and (\ref{3.25}), it follows that
  $R^2 - 2|\Ric|^2 \in L^1((0, T] \times M; dt \dvol_{g(t)})$. The
    corollary now follows from Theorem \ref{1.3}.
  \end{proof}

We now prove Theorem \ref{1.4}.
  There is a uniform existence time $[0,T]$ for the Ricci flow
  solutions $(M_i, g_i(t))$ with initial condition $g_i(0) = g_i$
  \cite{Bamler-Cabezas-Rivas-Wilking (2019)}.
  The flows have $2$-nonnegative curvature operator and satisfy
  $|\Rm_{g_i(t)}| < \frac{A}{t}$ for some $A < \infty$.
  By (\ref{3.25}) and Lemma \ref{4.4},
  $\int_{M_i} R_{g_i(t)} \: \dvol_{g_i(t)} \le \widehat{A}$.
  By Cheeger-Hamilton compactness, after passing to a subsequence there are
  \begin{enumerate}
  \item  A smooth manifold $X_\infty$,
  \item A Ricci flow solution $(X_\infty, g_\infty(t))$ defined for
    $t \in (0, T]$ and
    \item Diffeomorphisms $\phi_i : X_\infty
      \rightarrow M_i$
  \end{enumerate}
  so that for any $[a,b] \subset (0, T)$,
  $\lim_{i \rightarrow} \phi_i^* g_i = g_\infty$ smoothly on $[a,b]
  \times X_\infty$. Then $g_\infty(t)$ has $2$-nonnegative curvature
  operator, and satisfies $|\Rm_{g_\infty(t)}| < \frac{A}{t}$ and
  $\int_{X_\infty} R_{g_\infty(t)} \: \dvol_{g_\infty(t)} \le \widehat{A}$,
  for all $t \in (0, T]$.
  From Lemma \ref{3.1}, there is a Gromov-Hausdorff limit
  $\lim_{t \rightarrow 0} (X_\infty, g_\infty(t)) =
  (X_\infty, d_\infty)$.
  By Corollary \ref{4.10}, there is a limit
  $\lim_{t \rightarrow 0} R_{g_\infty(t)} \dvol_{g_\infty(t)} =
  \mu_0$ in the weak-$\star$ topology.

  The proof of Theorem \ref{1.3} can be made
  uniform in the underlying geometry.  Hence given $f \in C(X_\infty)$,
  it follows that
\begin{align} \label{4.13}
&  \lim_{i \rightarrow \infty} \int_{X_\infty} f \:
  \left( \phi_i^{-1} \right)_*
  \left( R_{g_i(t)} \dvol_{g_i(t)} \right) = \\
&     \lim_{i \rightarrow \infty} \int_{X_\infty} f \:
     R_{\phi_i^* g_i(t)} \dvol_{\phi_i^* g_i(t)} =
  \int_{X_\infty} f \:
  R_{g_\infty(t)} \dvol_{g_\infty(t)}, \notag
\end{align}
uniformly in $i$ and $t \in (0,T]$. 
Thus
\begin{equation} \label{4.14}
\lim_{i \rightarrow \infty} \int_{X_\infty} f \:
  \left( \phi_i^{-1} \right)_*
  \left( R_{g_i} \dvol_{g_i} \right) =
  \int_{X_\infty} f \: d\mu_0,
\end{equation}
which means that $\lim_{i \rightarrow \infty}
\left( \phi_i^{-1} \right)_*
  \left( R_{g_i} \dvol_{g_i} \right) \stackrel{weak-\star}{=} \mu_0$.

\begin{remark} \label{4.11}
There is a conjecture that
for any $n \in \Z^+$ and $v > 0$, there is some
  $\widehat{A} = \widehat{A}(n,v) < \infty$ so that if
$(M,g)$ is a complete $n$-dimensional Riemannian manifold
with $\Ric \ge - (n-1)g$, and $B(m,1)$ is a
unit ball in $M$ with
$\vol(B(m,1)) \ge v$,
  then $\int_{B(m,1)} R \: \dvol_g \le \widehat{A}$.
  This conjecture is known to be true if $M$ is a polarized
  K\"ahler manifold \cite[Proposition 1.7]{Liu-Szekelyhidi (2018)}
  or if $M$ has sectional
  curvature bounded below by $-1$ \cite{Petrunin (2009)}. 
  If the conjecture holds then condition (4) in Theorem \ref{1.4} follows
  from the first three conditions.
\end{remark}

       {\bf Addendum to ``On scalar curvature lower bounds and scalar curvature measure'' about (3.19)}

From Lemma 3.1, if $d_{\widehat{t}}(x,y) < L \sqrt{\widehat{t}}$ then
$d_T(x,y) < L  \sqrt{\widehat{t}} e^{E(\widehat{t}-T)}$. As
$$|f(x) - f(y)| \le d_T(x,y) \max |\nabla f|_{g(T)},$$
we obtain
$$    \max_{x \in B_{g(\widehat{t})}(y, L \sqrt{\widehat{t}})} |f(x) - f(y)|
      \le
      L \sqrt{\widehat{t}}
      e^{E(\widehat{t}-T)} \max |\nabla f|_{g(T)}.
$$
    The rest of the proof follows as in the paper. 
\end{document}